\title{Connectivity for random graphs from a weighted bridge-addable class}
\date{31 July 2012}
\documentclass[11pt]{article}
  \usepackage{amsmath}
  \usepackage{amssymb}
  \usepackage{latexsym}
  \usepackage{graphicx}

  \newenvironment{proof}{\noindent{\bf Proof} \hspace{.1in}}{\hspace*{\fill}$\Box$}
  \newenvironment{proofof}[1]{%
  \noindent {\bf Proof of #1}}%
  {\hspace*{\fill}$\Box$}

  \newtheorem{theorem}{Theorem}[section]
  \newtheorem{lemma} [theorem] {Lemma}
  \newtheorem{conjecture} [theorem] {Conjecture}

\let\eps=\epsilon
\def\enddiscard{}
\long\def\discard#1\enddiscard{}

  \def\Po{\mbox{\rm Po}}

  \newcommand{\E}{\mathbb E}
  \newcommand{\pr}{\mathbb P}

  \newcommand{\inu}{\in_{u}}
  
  \newcommand{\intau}{\in_{\tau}}

  \newcommand{\mass}[1]{\mbox{mass}(#1)}

  \newcommand{\cp}{{\mathcal P}}
  \newcommand{\cq}{{\mathcal Q}}
  
  \newcommand{\ca}{{\mathcal A}}
  \newcommand{\cb}{{\mathcal B}}
  \newcommand{\cc}{{\mathcal C}}

  \newcommand{\cf}{{\mathcal F}}
  \newcommand{\cg}{{\mathcal G}}
  \newcommand{\ct}{{\mathcal T}}

  \newcommand{\RF}{R^{\cf}}
  \newcommand{\RT}{R^{\ct}}
  \newcommand{\RG}{R^{G_0}}

  \newcommand{\frag}{\mbox{\rm frag}}
  \newcommand{\wfrag}{\mbox{\rm wfrag}}

  \newcommand{\Bridge}{\rm Bridge}
  
  \newcommand{\Cross}{\rm Cross}
  \newcommand{\cross}{\rm cross}
  
  \newcommand{\ttau}{\tilde{\tau}}

  \newcommand{\Go}{G^{o}}
  \newcommand{\cao}{\ca^{o}}

\begin{document}

\author{Colin McDiarmid \\ Department of Statistics \\ Oxford University}
\maketitle

\begin{abstract}
  There has been much recent interest in random graphs sampled uniformly from the
  $n$-vertex graphs in a suitable structured class, such as the class of all planar graphs.
  Here we consider a general \emph{bridge-addable} class $\ca$ of graphs --
  if a graph is in $\ca$ and $u$ and $v$ are vertices in different components
  then the graph obtained by adding an edge (bridge) between $u$ and $v$ must also be in $\ca$. 
  Various bounds are known concerning the probability of a random graph from such a
  class being connected or having many components, sometimes under the additional assumption that bridges
  can be deleted as well as added.
  Here we improve or amplify or generalise these bounds (though we do not resolve the main conjecture).
  For example, we see that the expected number of vertices left when we remove a largest component is less than $2$.
  The generalisation is to consider `weighted' random graphs,
  sampled from a suitable more general distribution,
  where the focus is on the bridges.
  %
\end{abstract}


\section{Introduction}
\label{sec.intro}

  A \emph{bridge} in a graph is an edge $e$ such that the graph $G \setminus e$ obtained by deleting $e$ has
  one more component. 
  A class $\ca$ of graphs is {\em bridge-addable} if for all graphs $G$ in $\ca$ and all vertices
  $u$ and $v$ in distinct connected components of $G$, the graph $G +uv$ obtained by adding an
  edge between $u$ and $v$ is also in $\ca$.
  The concept of being bridge-addable (or {`weakly addable'}) was introduced in 
  McDiarmid, Steger and Welsh~\cite{msw05} in the course of studying random planar graphs.
  (For an overview on random planar graphs see the survey paper~\cite{gn09b} of Gim\'enez and Noy.)
  Examples of bridge-addable classes of graphs include forests, series-parallel graphs,
  planar graphs, and indeed graphs embeddable on any given surface.
  
  In the rest of this section, we first describe what is known concerning connectedness and components for
  random graphs sampled uniformly from a bridge-addable class;
  then describe the new results here for such random graphs; and finally briefly discuss random rooted graphs.
  Random graphs from a weighted class are introduced in Section~\ref{sec.rwg}; and new general results are presented,
  which extend the results on uniform random graphs.
  After that come the proofs, first for non-asymptotic results then for asymptotic results and finally for the rooted case.  
  \medskip
  
  \noindent
  {\em Background results for uniform random graphs}
  
  If $\ca$ is finite and non-empty we write $R \inu \ca$ to mean that $R$ is a random graph sampled
  {\bf u}niformly from $\ca$.  (We consider graphs to be labelled.)
  The basic result on connectivity for a bridge-addable set of graphs is
  Theorem 2.2 of~\cite{msw05}:
  if $\ca$ is a finite bridge-addable set of graphs and $R \in_u \ca$ then
\begin{equation} \label{eqn.b-a-conn}
  \pr(R \mbox{ is connected}) \geq e^{-1} .
\end{equation}
  Indeed, a stronger result is given in~\cite{msw05}, concerning the number $\kappa(R)$ of components of $R$;
  namely that $\kappa(R)$ is stochastically at most $1+\Po(1)$
  where $\Po(\lambda)$ denotes a Poisson-distributed random variable with mean $\lambda$,
  that is
\begin{equation} \label{eqn.b-a-comps}
  \kappa(R) \leq_s 1+ \Po(1).
\end{equation}
  (Recall that $X \leq_s Y$ means that $\pr(X \leq t) \geq \pr(Y \leq t)$ for each $t$.)
  Note that from~(\ref{eqn.b-a-comps}) we have
\[   \pr(R \mbox{ is connected}) = \pr(\kappa(R) \leq 1) \geq \pr(\Po(1) \leq 0) = e^{-1}\]
  and we obtain~(\ref{eqn.b-a-conn}).
  Also $\E[\kappa(R)] \leq 2$ (see~(\ref{eqn.exp}) below).

  For any set $\ca$ of graphs, we let $\ca_n$ denote the set of graphs in $\ca$ on vertex set $[n]:=\{1,\ldots,n\}$;
  and we write $R_n \inu \ca$ to mean that $R_n$ is uniformly distributed over $\ca_n$.
  We always assume that $\ca_n$ is non-empty at least for large $n$.
  
  The class $\cf$ of forests is of course bridge-addable. For $R_n \in_u \cf$ a result of R\'enyi~\cite{renyi59} shows that
  $\pr(R_n \mbox{ is connected}) \to e^{-\frac12}$ as $n \to \infty$, and indeed
  $\kappa(R_n)$ converges in distribution to $1+ \Po(\frac12)$.
  For background on random trees and forests see the books~\cite{drmota2009,moon70}.
  It was noted in~\cite{msw06} that plausibly forests form the `least connected' bridge-addable set of graphs,
  and in particular it should be possible to improve the bound in~(\ref{eqn.b-a-conn}) asymptotically.
\begin{conjecture} \cite{msw06} \label{conj.b-add}
  If $\ca$ is bridge-addable and $R_n \in_u \ca$ then
\begin{equation} \label{eqn.b-a_conj}
  \liminf_{n \to \infty} \pr(R_n \mbox{ is connected}) \geq e^{-\frac12}.
\end{equation}
\end{conjecture}

  Balister, Bollob{\'a}s and Gerke~\cite{bbg07,bbg10} showed that inequality~(\ref{eqn.b-a_conj}) holds
  if we replace $e^{-\frac12} \approx 0.6065$ by the weaker bound $e^{-0.7983} \approx 0.4542$.
  This result has recently been improved by Norine~\cite{sn2012}. 
  Recently Addario-Berry, McDiarmid and Reed~\cite{amr2011}, and Kang and Panagiotou~\cite{kp2011},
  separately showed that~(\ref{eqn.b-a_conj}) holds with the desired lower
  bound $e^{-\frac12}$, if we suitably strengthen the condition on $\ca$. 
  Call a set $\ca$ of graphs {\em bridge-alterable} if it is bridge-addable and also closed under
  deleting bridges. Thus $\ca$ is bridge-alterable exactly when it satisfies the condition that,
  for each graph $G$ and bridge $e$ in $G$, the graph $G$ is in $\ca$ if and only if $G \setminus e$ is in $\ca$.
  Observe that each of the bridge-addable classes of graphs mentioned above is in fact bridge-alterable.
  If $\ca$ is a bridge-alterable set of graphs and $R_n \in_u \ca$ then~\cite{amr2011,kp2011} 
\begin{equation} \label{eqn.b-alt}
  \liminf_{n \to \infty} \pr(R_n \mbox{ is connected}) \geq e^{-\frac12}.
\end{equation}
  Since the class $\cf$ of forests is bridge-alterable, this result is best-possible for a bridge-alterable
  set of graphs.  The full version of Conjecture~\ref{conj.b-add} (for a bridge-addable set) is still open.

  Next let us consider the `fragment' of a graph $G$: 
  we let $\frag(G)$ be the number of vertices remaining when we remove a largest component.
  For the class $\cf$ of forests, if $R_n \in_u \cf$ then
\begin{equation} \label{eqn.forests-frag}  
  \E[\frag(R_n)] \to 1 \;\; \mbox{ as } n \to \infty.
\end{equation}
  It was shown in~\cite{cmcd09} that,
  if $\ca$ is a bridge-addable class of graphs which satisfies the further condition
  that it is closed under forming minors (and so $\ca$ is bridge-alterable), 
  then there is a constant $c=c(\ca)$ such that, for $R_n \in_u \ca$ 
\begin{equation} \label{eqn.b-alt-frag1}
  \E[\frag(R_n)] \leq c \;\; \mbox{ for all } n.
\end{equation}
  \smallskip
  
  \noindent
  {\em New results for uniform random graphs}
  
  In the present paper we much improve inequality~(\ref{eqn.b-alt-frag1}) and extend all the above results to more
  general distributions (similar to distributions considered in~\cite{cmcd-rwg2012}),
  though we continue to consider uniform random graphs in this section.
  (All the results presented here are special cases of results discussed in the following section.)
  
  In particular we see that, if $\ca$ is \emph{any} bridge-addable class of graphs
  (with no further conditions) and $R_n \in_u \ca$, then
\begin{equation} \label{eqn.b-add-frag}
  \E[\frag(R_n)] < 2 \;\; \mbox{ for all } n;
  \end{equation}
  and if $\ca$ is bridge-alterable then 
\begin{equation} \label{eqn.b-alt-frag2}
  \limsup_{n \to \infty} \E[\frag(R_n)] \leq 1.
\end{equation}
  Observe from the limiting result~(\ref{eqn.forests-frag}) that this last bound is optimal for a bridge-alterable
  set of graphs, but perhaps it holds for any bridge-addable set of graphs-- see Section~\ref{sec.concl}.

  We also strengthen inequality~(\ref{eqn.b-alt}) in much the same way that~(\ref{eqn.b-a-comps})
  strengthens~(\ref{eqn.b-a-conn}).
  Given non-negative integer-valued random variables $X_1,X_2,\ldots$ and $Y$, we say that
  $X_n$ is \emph{stochastically at most} $Y$ \emph{asymptotically}, and write $X_n \leq_s Y$ asymptotically,
  if for each fixed $t \geq 0$,
\[  \limsup_{n \to \infty} \pr(X_n \geq t) \leq \pr(Y \geq t). \]
 Our strengthening of~(\ref{eqn.b-alt}) is that, if $\ca$ is bridge-alterable and $R_n \inu \ca$, then 
\begin{equation} \label{eqn.b-alt-comps1}
  \kappa(R_n) \leq_s 1+ \Po(\frac{1}{2}) \;\; \mbox{ asymptotically }.
\end{equation}
  
  \medskip
  
\noindent
{\em Random rooted graphs}
  
  It may be enlightening to consider rooted graphs.
  We say that a graph is {\em rooted} if each component has a specified root vertex.
  We will use the notation $\Go$ for a rooted graph;
  and given a class $\ca$ of graphs we write $\cao$ for the corresponding class of rooted graphs.
  Thus a connected graph in $\ca_n$ yields $n$ rooted graphs in the corresponding set $\cao_n$;
  a graph in $\ca_n$ which has two components, with respectively $a$ and $n-a$ vertices, yields $a(n-a)$ rooted graphs
  in $\cao_n$; and so on.
  We use the notations $R^{o} \inu \cao$ and $R_n^{o} \inu \cao$ as before, to indicate that
  $R^{o}$ is sampled uniformly from $\cao$ (assumed finite) and $R_n^{o}$ is uniformly sampled from $\cao_n$.
  
  Now let $\ca$ be a finite bridge-addable set of graphs, and let $R^{o} \inu \cao$.
  Since a graph with several non-singleton components generates many rooted graphs,
  it is not immediately clear to what extent the earlier results on connectedness and components will survive.
  We will see that the analogues of~(\ref{eqn.b-a-conn}) and~(\ref{eqn.b-a-comps}) both hold:
\begin{equation} \label{eqn.rooted-conn}
  \pr(R^{o} \mbox{ is connected}) \geq e^{-1}
\end{equation}
  and indeed
\begin{equation} \label{eqn.rooted-comps}
  \kappa(R^{o}) \leq_s 1+ \Po(1).
\end{equation}
  Now consider the class $\cf$ of forests, and let $R_n^{o} \inu \cf^{o}$.
  Then as $n \to \infty$
\[ \pr(R_n^{o} \mbox{ is connected}) = (\frac{n}{n+1})^{n-1} \to e^{-1}\]
  and indeed $\kappa(R_n^{o})$ converges in distribution to $1+\Po(1)$.
  Thus~(\ref{eqn.rooted-conn}) and~(\ref{eqn.rooted-comps}) are best possible, in contrast to the unrooted case.
  Further, 
  $\E[\frag(R_n^{o})] \to \infty$ as $n \to \infty$,
  so there is no analogue for~(\ref{eqn.b-add-frag}) or~(\ref{eqn.b-alt-frag2}) for rooted graphs.
  \bigskip

  In all these results the crucial feature is the behaviour of the bridges.
  We shall bring this out by singling out bridges in the more general distributions we next
  introduce for our random graphs.


\section{Random weighted graphs} \label{sec.rwg}
 
  Given a graph $G$ with vertex set $V$, let $e(G)$ denote the number of edges,
  let $e_0(G)$ denote the number of bridges (edges in 0 cycles) and let
  $\tilde{G}$ denote the graph on $V$ obtained from $G$ by removing all bridges.
  Thus $\kappa(\tilde{G})= \kappa(G)+e_0(G)$.
  
  %
  Let $\lambda >0$ and $\nu >0$, and let $f(G) \geq 0$ for each bridge-free graph $G$.
  We call $(\lambda,\nu,f)$ a \emph{weighting} and define the \emph{weight} $\tau(G)$ of $G$
  by setting
\begin{equation} \label{eqn.probdef}
  \tau(G) = f(\tilde{G}) \, \lambda^{e_0(G)} \nu^{\kappa(G)}.
\end{equation}
  Given a set $\ca$ of graphs, let $\tau(\ca)$ denote $\sum_{G \in \ca} \tau(G)$.
  When $0<\tau(\ca) < \infty$, we let $R \in_{\tau} \ca$ mean that $R$ is a random graph sampled from $\ca$ with
  $\pr(R=G) = \tau(G)/ \tau(\ca)$ for each graph $G \in \ca$.  Similarly
  $R_n \in_{\tau} \ca$ means that $R_n$ is a random graph sampled from $\ca_n$ with
  $\pr(R=G) = \tau(G)/ \tau(\ca_n)$ for each graph $G \in \ca_n$ (and we assume that $\tau(\ca_n)>0$).
  In the special case when $\lambda=\nu=1$ and $f(G) \equiv 1$, clearly $R \in_{\tau} \ca$ and
  $R_n \in_{\tau} \ca$ mean the same as $R \inu \ca$ and $R_n \inu \ca$ respectively.
  When $f(G)\equiv \lambda^{e(G)}$ we have $\tau(G) = \lambda^{e(G)} \nu^{\kappa(G)}$,
  and we do not single out bridges.
  
  Recall that the classical Erd\H{o}s-R\'enyi (or binomial) random graph $G_{n,p}$ has vertex set $[n]$,
  and the ${n \choose 2}$ possible edges are included independently with probability $p$,
  where $0<p<1$. Assuming that $\ca_n$ is non-empty, for each $H \in \ca_n$ we have
  \[
    \pr(G_{n,p}=H | G_{n,p} \in \ca ) =
    \frac{p^{e(H)}(1-p)^{{n \choose 2}-e(H)}}
    {\sum_{G \in \ca_n}p^{e(G)}(1-p)^{{n \choose 2}-e(G)}}
  =
   \frac{\lambda^{e(H)}}
    {\sum_{G \in \ca_n}\lambda^{e(G)}}
\]
  where $\lambda=p/(1-p)$.
  %
  Now consider the more general random-cluster model (see for example~\cite{grimmett06}),
  where we are also given $\nu>0$ (we use $\nu$ rather than $q$), and
  the random graph $R_n$ takes as values the graphs $H$ on $[n]$, with
\[
  \pr(R_n=H) \propto p^{e(H)}(1-p)^{{n \choose 2}-e(H)} \nu^{\kappa(H)}.
\]
  Then for each $H \in \ca_n$
  \[
    \pr(R_n=H | R_n \in \ca )
  =
   \frac{\lambda^{e(H)}\nu^{\kappa(H)}}
    {\sum_{G \in \ca_n}\lambda^{e(G)} \nu^{\kappa(G)}} = \frac{\tau(H)}{\tau(\ca_n)}
\]
  where $\tau(H)= \lambda^{e(H)} \nu^{\kappa(H)}$, as we met above.
  \medskip
  
  
  Suppose now that we are given a set $\ca$ of graphs and a weighting $(\lambda,\nu,f)$,
  and that $0<\tau(\ca)<\infty$ or $\tau(\ca_n) >0$ as appropriate. 
  We generalise and sometimes amplify all the results presented in the last section.
  For the asymptotic results 
  we need to assume that $\ca$ is bridge-alterable rather than just bridge-addable.
  
  We first state two non-asymptotic results; then present some results on random forests,
  and consider asymptotic results; and finally we consider random rooted graphs.
  The first result generalises the inequalities~(\ref{eqn.b-a-conn}) and~(\ref{eqn.b-a-comps}), and 
  is used several times in~\cite{cmcd-rwg2012}; and the second result generalises inequality~(\ref{eqn.b-add-frag}).

\begin{theorem} \label{prop.tauconn}
  If $\ca$ is finite and bridge-addable and $R \in_{\tau} \ca$, then
\[ \kappa(R) \leq_s 1+ \Po(\nu/\lambda); \]
  and in particular $\; \pr(R \mbox{ is connected}) \geq e^{-\nu/\lambda}$, and
  $\E[\kappa(R)] \leq 1+ \nu/\lambda$.
\end{theorem}
  %
%



\begin{theorem} \label{prop.fragbound}
  If $\ca$ is finite and bridge-addable and $R \in_{\tau} \ca$, then
\[ \E[\frag(R)] < \frac{2\nu}{\lambda}. \]
\end{theorem}

  


  Before we introduce the asymptotic results for a general bridge-alterable set of graphs,
  let us record some results on random forests $R_n \in_{\tau} \cf$ which generalise the results
  mentioned earlier for uniform random forests $R_n \inu \cf$ -- 
  see for example~\cite{cmcd-rwg2012} where these results are proved in a general setting.
  Observe that $\tau(F)= f(\bar{K_n}) (\lambda/\nu)^{e(F)} \nu^{n}$ for each $F \in \cf_n$
  (where $\bar{K_n}$ denotes the graph on $[n]$ with no edges):
  thus $\tau(F) \propto (\lambda/\nu)^{e(F)}$,
  and the only aspect of the weighting that matters is the ratio $\lambda/\nu$.
\begin{theorem}\label{prop.asympt-forests}
  Consider $R_n \in_{\tau} \cf$, where $\cf$ is the class of forests.
  Then $\; \kappa(R_n)$ converges in distribution to $1+ \Po(\frac{\nu}{2\lambda})$, so
  $\pr(R_n \mbox{ is connected}) \to e^{-\frac{\nu}{2 \lambda}}$;
  $\E[\kappa(R_n)] \to 1 + \frac{\nu}{2 \lambda}$ as $n \to \infty$; and 
  $\E[ \frag(R_n)] \to \frac{\nu}{\lambda}$ as $n \to \infty$.
\end{theorem}
  \noindent
  Now we consider asymptotic results for a bridge-alterable set of graphs.
  These results generalise and amplify inequalities~(\ref{eqn.b-alt}) and~(\ref{eqn.b-alt-frag2});
  and Theorem~\ref{prop.asympt-forests} shows that each of inequalities~(\ref{eqn.asympt-comps})
  to~(\ref{eqn.asympt-frag}) is best-possible for a bridge-alterable class of graphs.
  %
\begin{theorem}\label{prop.asympt-conn}
  Suppose that $\ca$ is bridge-alterable 
  and  $R_n \intau \ca$.  Then
\begin{equation} \label{eqn.asympt-comps}
  \kappa(R_n) \leq_s 1+ \Po(\frac{\nu}{2\lambda}) \;\; \mbox{ asymptotically},
\end{equation}
  and so in particular
\begin{equation} \label{eqn.asympt-conn}
     \liminf_{n \to \infty} \pr(R_n \mbox{ is connected}) \geq e^{-\frac{\nu}{2\lambda}};
\end{equation}
  and
\begin{equation} \label{eqn.asympt-Ekappa}
 \limsup_{n \to \infty} \E[\kappa(R_n)] \leq 1+ \frac{\nu}{2\lambda}.
\end{equation}
\end{theorem}
  %
\begin{theorem} \label{prop.fragbound2}
  If $\ca$ is bridge-alterable 
  and $R_n \in_{\tau} \ca$, then
\begin{equation} \label{eqn.asympt-frag}
  \limsup_{n \to \infty} \E[\frag(R_n)] \leq \frac{\nu}{\lambda}.
\end{equation}
\end{theorem}

\medskip

  Now consider rooted graphs, starting with rooted forests.  Recall that $\cf^{o}$ denotes the class of rooted forests.
  
\begin{theorem}\label{prop.rooted-forests}
  Consider $R^{o}_n \in_{\tau} \cf^{o}$. 
  As $n \to \infty$, 
  $\kappa(R^{o}_n)$ converges in distribution to $1+\Po(\frac{\nu}{\lambda})$;
  and so
  $\pr(R_n^{o} \mbox{ is connected}) \to e^{-\frac{\nu}{\lambda}}$, and
  $\E[\kappa(R_n^{o})] \to 1 + \frac{\nu}{\lambda}$ as $n \to \infty$.
  In contrast, $\E[ \frag(R_n^{o})] \to \infty$ as $n \to \infty$.
\end{theorem}  
  
    Our final result here is non-asymptotic and may be compared with
  Theorem~\ref{prop.tauconn}.
  It generalises~(\ref{eqn.rooted-conn}) and~(\ref{eqn.rooted-comps}).
  Theorem~\ref{prop.rooted-forests} on rooted forests shows that it is best possible,
  and that there is no rooted-graph analogue for Theorem~\ref{prop.fragbound} (which bounds $\E[ \frag(R_n)]$).

\begin{theorem} \label{prop.rooted}
  Let $\ca$ be finite and bridge-addable, and let $R^{o} \in_{\tau} \cao$.
  Then
\[ \kappa(R^{o}) \leq_s 1+ \Po(\nu/\lambda); \]
  and in particular $\; \pr(R^{o} \mbox{ is connected}) \geq e^{-\nu/\lambda}$, and
  $\E[\kappa(R^{o})] \leq 1+ \nu/\lambda$.
\end{theorem}


\section{Proofs for non-asymptotic results}  
\label{sec.proof-non-a}

  In this section we prove the non-asymptotic results above, namely Theorems~\ref{prop.tauconn},
  \ref{prop.fragbound}, except that we leave proofs for rooted graphs to Section~\ref{sec.rooted}. 
  
  Given a graph $G$, let $\Bridge(G)$ denote the set of bridges, so that $|\Bridge(G)|=e_0(G)$;
  and let $\Cross(G)$ denote the set of `non-edges' or `possible edges' between components, and let
  $\cross(G)=|\Cross(G)|$.  
  We start with two basic lemmas about graphs.
  The first is just an observation, and needs no proof.
  
\begin{lemma} \label{lem.basic1}
  Let the set $\ca$ of graphs be bridge-addable. 
  If $G \in \ca$ and $e \in \Cross(G)$,
  then the graph $G' = G +e$ obtained from $G$ by adding $e$ is in $\ca$ and $e$ is a bridge of $G'$;
  and $\tau(G) = \tau(G') \cdot (\nu/\lambda)$.
\end{lemma}
  %
  %
\begin{lemma} \label{lem.basic2} \cite{msw05}
  If the graph $G$ has $n$ vertices, then $e_0(G) \leq n - \kappa(G)$;
  and if $\kappa(G)=k+1$ then $\cross(G) \geq k(n-k) + {k \choose 2} \geq k(n-k)$.  
\end{lemma}  
  
\begin{proof}
  Observe that $\kappa(G) + e_0(G) = \kappa(\tilde{G}) \leq n$, so $e_0(G) \leq n-\kappa(G)$.
  Now consider the second inequality, and assume that $\kappa(G)=k+1$.
  Since if $0<|X| \leq |Y|$ then $|X| |Y| >(|X|-1)(|Y|+1)$, we see that $\cross(G)$ is minimised when
  $G$ consists of $k$ singleton components and one other component.
\end{proof}  
\medskip

  Now let us recall a well-known elementary fact.
  Let $X$ and $Y$ be random variables taking non-negative integer values, and suppose that $X \leq_s Y$: then
\begin{equation} \label{eqn.exp}
  \E X \leq \E Y.
\end{equation}
  To prove this, note that
\[ \E X = \sum_{t \geq 1} \pr(X \geq t) \leq \sum_{t \geq 1} \pr(Y \geq t) = \E Y. \]

  The next two lemmas concern bounding a random variable by a Poisson-distributed random variable.
  The first lemma is stated in a general form which quickly gives the second and which is 
  suitable also for using later.  
\begin{lemma} \label{lem.po1}
  Let the random variable $X$ take non-negative integer values. 
  Let $\alpha>0$ and let $Y \sim \Po(\alpha)$. Let $k_0$ be a positive integer, and suppose that
\[ \pr(X=k+1) \leq \frac{\alpha}{k+1} \, \pr(X=k) \;\; \mbox{ for each } k=0,1,\ldots,k_0-1. \]
  Then
\[ \pr(k_0 \geq X \geq k) \leq \pr(Y \geq k) \;\; \mbox{ for each } k=0,1,\ldots,k_0 .\]
\end{lemma}
\smallskip

\begin{proof}
  Observe that, for each $k=0, 1, \ldots, k_0-1$ we have
\[ \pr(X=k+i) \leq \frac{\alpha^i}{(k+i)_i} \, \pr(X=k) \;\; \mbox{ for each } i=1,\ldots,k_0 -k. \]
  Clearly $\pr(k_0 \geq X \geq 0) \leq 1 = \pr(Y \geq 0)$.
  Let $k_0 > k \geq 0$ and suppose that $\pr(k_0 \geq X \geq k) \leq \pr(Y \geq k)$.
  We want to show that $\pr(k_0 \geq X \geq k+1) \leq \pr(Y \geq k+1)$
  to complete the proof by induction.  This is immediate if
  $\pr(X = k) \geq \pr(Y = k)$, so assume that this is not the case. Then 
 \begin{eqnarray*}
   \pr(k_0 \geq X \geq k+1)
  &=&
    \sum_{i=1}^{k_0-k} \pr(X=k+i) \\
  &\leq &
   \pr(X=k) \sum_{i \geq 1} \frac{\alpha^i}{(k+i)_i} \\
  & \leq &
   \pr(Y=k) \sum_{i \geq 1} \frac{\alpha^i}{(k+i)_i} \\
  & = &
   \pr(Y \geq k+1)
\end{eqnarray*}
  as required.
\end{proof}  
\medskip

  \noindent
  From the last lemma with $k_0$ large we obtain
\begin{lemma} \label{lem.po2} (see~\cite{msw05})
  Let the random variable $X$ take non-negative integer values. 
  Let $\alpha>0$ and suppose that
\[ \pr(X=k+1) \leq \frac{\alpha}{k+1} \, \pr(X=k) \;\; \mbox{ for each } k=0,1,2,\ldots \]
  Then $X \leq_s Y$
  where $Y \sim \Po(\alpha)$.
\end{lemma}

\begin{proof}
  Fix $k \geq 0$.  Let $\eps>0$ and choose $k_0 \geq k$ such that $\pr(X>k_0)<\eps$.
  By Lemma~\ref{lem.po1}
\[ \pr(X \geq k) = \pr(k_0 \geq X \geq k) + \pr(X>k_0) \leq \pr(Y \geq k)+ \eps,\]
  and thus $ \pr(X \geq k) \leq \pr(Y \geq k)$.
\end{proof}
\medskip

\begin{proofof} {Theorem~\ref{prop.tauconn}}
  %
  %
  %
  It suffices to assume that $\ca$ is $\ca_n$ for some $n$, since the sets
  $\ca_n$ are disjoint. Let $\ca_n^{k}$ denote the set of graphs in $\ca_n$ with $k$ components.
   Let $1 \leq k \leq n-1$.
  By Lemmas~\ref{lem.basic1} and~\ref{lem.basic2}
\begin{eqnarray*}
 \tau(\ca_n^{k}) \cdot(n-k)
 & \geq &
  \sum_{G,e} \{ \tau(G): G \in \ca_n^{k}, e \in \Bridge(G)\}\\
 & \geq &
  \sum_{H,e} \{ \tau(H): H \in \ca_n^{k+1}, e \in \Cross(H) \} \cdot (\lambda/\nu)\\
 & \geq &
   \tau(\ca_n^{k+1}) \cdot k(n-k) \cdot (\lambda/\nu).
\end{eqnarray*}
    Therefore
\[   \tau(\ca_n^{k+1}) \leq \frac{\nu}{\lambda k} \  \tau(\ca_n^{k}). \]
%
  %
  %
  Thus for $R \in_{\tau} \ca$
\[
  \pr(\kappa(R) = k+1) \leq \frac{\nu}{\lambda k} \pr(\kappa(R) = k) \;\; \mbox{ for each } k=1,2,\ldots.
\]
  and so, writing $X=\kappa(R)-1$
\[
  \pr(X = k+1) \leq \frac{\nu}{\lambda (k+1)} \pr(X = k) \;\; \mbox{ for each } k=0,1,2,\ldots,
\]
  %
  Hence if $\alpha= \nu/\lambda$ and $Y \sim\Po(\alpha)$ we have  $X \leq_s Y$ by Lemma~\ref{lem.po2}.
  Finally, 
  by~(\ref{eqn.exp}), $\; \E[\kappa(R)] = 1+\E[X] \leq 1+ \E[Y] = 1 + \nu/\lambda$.
\end{proofof}
  
  %
  %
  %
  %
  \bigskip  

  
  To prove Theorem~\ref{prop.fragbound} we use two lemmas.  The first is another basic lemma on graphs.
\begin{lemma} \label{lem.basic3} \cite{cmcd08}
  If the graph $G$ has $n$ vertices, then $\cross(G) \geq (n/2) \cdot \frag(G)$.  
\end{lemma}   
\begin{proof}
  An easy convexity argument shows that if $x,x_1,x_2,\ldots$ are
  positive integers such that each $x_i \leq x$ and $\sum_i x_i=n$
  then $\sum_i {x_i \choose 2} \leq \frac12 n (x-1)$.
  For, if $n=ax+y$ where $a \geq 0$ and $0 \leq y \leq x-1$ are integers, then
  \[ \sum_i {x_i \choose 2} \leq a {x \choose 2} + {y \choose 2}
  \leq a {x \choose 2} + \frac{y(x-1)}{2} = \frac12 n (x-1).\]
  Hence if we denote the maximum number of vertices in a component by $x$,
  so that $\frag(G) = n-x$, then 
  \[\cross(G) \geq {n \choose 2} - \frac12 n (x-1)
  = \frac12 n (n-x) = \frac12 n \ \frag(G)\]
  as required.
\end{proof}
\medskip
  
  The next lemma is phrased generally so that it can also be used later. 
  
\begin{lemma} \label{lem.frag1}
  Let $\ca=\ca_n$ be bridge-addable, and let $R \in_{\tau} \ca$.
  Let $\beta>0$ and assume that $\cross(G) \geq \beta n \cdot \frag(G)$ for each $G \in \ca$.
  Then
\[  \E [\frag(R)] \leq \frac{\nu}{\beta \lambda}.  \]  
\end{lemma}
  
\begin{proof}
  Using Lemma~\ref{lem.basic1} 
\begin{eqnarray*}
 && \beta n \, \sum_{G \in \ca} \tau(G) \, \frag(G)\\
 & \leq &
   \sum_{G, e} \{ \tau(G): G \in \ca, e \in \Cross(G) \} \\
& \leq &
   \frac{\nu}{\lambda} \ \sum_{G', e} \{ \tau(G'): G' \in \ca, e \in \Bridge(G')\}\\
 & = &
   \frac{\nu}{\lambda} \ \sum_{G \in \ca} \tau(G) \cdot e_0(G).
\end{eqnarray*}
  Thus
\[
  \E[\frag(R)]  \leq \frac{1}{ \beta n} \, \frac{\nu}{\lambda} \, \E[e_0(R)] <  \frac{\nu}{\beta \lambda}
\]
  as required.
\end{proof}
  \medskip

\begin{proofof} {Theorem~\ref{prop.fragbound}}
  As in the proof of Theorem~\ref{prop.tauconn}
  it suffices to assume that $\ca$ is $\ca_n$ for some $n$.
  %
  By Lemma~\ref{lem.basic3}, we may now complete the proof using Lemma~\ref{lem.frag1} with $\beta= \frac12$.
\end{proofof}
\bigskip


\section{Proofs of asymptotic results} 
\label{sec.asympt1}

  In this section we prove the asymptotic results 
  Theorems~\ref{prop.asympt-conn} and~\ref{prop.fragbound2}.
  Assume throughout that $\ca$ is bridge-alterable. 
  Let us focus first on 
  Theorem~\ref{prop.asympt-conn}, and in particular on~(\ref{eqn.asympt-comps}).
  The proof goes roughly as follows.
  We first see that it suffices to prove inequality~(\ref{eqn.impliesmain3}) below concerning $\kappa(\RF)$,
  where $\RF$ is a random forest on $[n]$ which we define below, with probabilities
  depending on degrees.
  Then we use a key result from~\cite{amr2011} which tells us about average sizes of components
  of $\RT$ with an edge deleted, where $\RT$ is $\RF$ conditioned on being a tree.
  We find that $\pr(\kappa(\RF)=2)$ is suitably smaller than $\pr(\kappa(\RF)=1)$;
  and from this we deduce that in general $\pr(\kappa(\RF)=k+1)$ is suitably smaller than $\pr(\kappa(\RF)=k)$,
  and so we can use Lemma~\ref{lem.po1}.

  Now for more details.
  We may define an equivalence relation on graphs by setting $G \sim H$ if $\tilde{G}=\tilde{H}$.
  Let $[G]$ denote the equivalence class of $G$, that is, the set of graphs $H$ such that
  $\tilde{H}=\tilde{G}$.
  Let $W$ be a positive integer.
  Since $\ca$ is bridge-alterable, if $G \in \ca_W$ then $[G] \subseteq \ca_W$.
  Thus $\ca_W$ can be written as a disjoint union of equivalence classes.
  
  To prove~(\ref{eqn.asympt-comps}), 
  we may fix a (large) positive integer $W$,
  a bridgeless graph $G_0 \in \ca_W$, an integer $t \geq 1$ and real $\eps>0$; and prove that,
  if $\RG \intau [G_0]$ then
\begin{equation} \label{eqn.impliesmain2}
  \pr(\kappa(\RG) \geq t+1) \leq \pr(\Po(\frac{\nu}{2\lambda}) \geq t) + \eps
\end{equation}
  if $W$ is sufficiently large. 
  Since we are now restricting attention to $[G_0]$ 
  we may assume that $f(G_0)= 1$.  Denote $\kappa(G_0)$ by $n$: we may assume that $n \geq 2$
  (for otherwise the connected graph $G_0$ is the only graph in $[G_0]$).

  Write $C_1,\ldots,C_n$ for the components of $G_0$, and let $w_i=|V(C_i)|$
  for $i=1,\ldots,n$, so that $W = \sum_{i=1}^n w_i$.
  We use the vector $w=(w_1,\ldots,w_n)$ together with the weighting $\tau$ to define a probability measure on the
  set $\cf_n$ of forests on~$[n]$. 
  Given $F \in \cf_n$, let
\[ \mass{F} = \prod_{i=1}^n w_i^{d_F(i)} \cdot \lambda^{e(F)} \nu^{\kappa(F)} \]
  where $d_F(i)$ denotes the degree of vertex $i$ in the forest $F$.
  Also, let $K = \sum_{F \in \cf_n} \mass{F}$, and let $\RF$ be a random element of $\cf_n$ with
  $\pr(\RF =F) = \mass{F}/K$ for each $F \in \cf_n$.
%
  Corresponding to Lemma 2.3 of~\cite{amr2011}, we have
\begin{equation} \label{eqn.stocheq}
  \kappa(\RG) =_s \kappa(\RF).
\end{equation}
  %
\begin{proofof}{(\ref{eqn.stocheq})}
    Denote $[G_0]$ by $\cb$.
    Given $H \in \cb$, let $g(H)$ be the graph obtained from $H$ by contracting
    each $C_i$ to the single vertex $i$. 
    Then $g(H)\in \cf_n$ and $\kappa(H)= \kappa(g(H))$.  Also,
    for each $F \in \cf_n$, the set $g^{-1}(F)$ has cardinality
    $\prod_{i=1}^n w_i^{d_F(i)}$, and so $\tau(g^{-1}(F))= \mass{F}$.
   It follows that
\begin{eqnarray*}
   \pr(\kappa(\RG)=k) & = & \frac{\tau(\{H\in\cb : \kappa(H)=k \})}{\tau(\cb)}\\
      & = & \frac{ \sum_{\{F\in\cf_n: \kappa(F)=k \}}\tau(g^{-1}(F))}{\sum_{F\in\cf_n}\tau(g^{-1}(F))} \\
      & = & \frac{ \sum_{ \{ F\in\cf_n: \kappa(F)=k \} } \mass{F} }{K} \\
      & = & \pr(\kappa(\RF)=k)
\end{eqnarray*}
  as required.
\end{proofof}
\smallskip

  Now that we have established~(\ref{eqn.stocheq}), in order to prove~(\ref{eqn.impliesmain2}) we may show that
\begin{equation} \label{eqn.impliesmain3}
  \pr(\kappa(\RF) \geq t+1) \leq \pr(\Po(\frac{\nu}{2\lambda}) \geq t) + \eps
\end{equation}
  if $W$ is sufficiently large.

 %
  Given a graph $H$ on $[n]$, let
\[ \cross_w(H) = \sum_{uv \in \Cross(H)} w_u w_v. \]
  Observe that $\cross_w(H)$ equals 
  the sum of $w(C) w(C')$ over the unordered pairs $C$ and $C'$ of
  components of $H$, where $w(C)$ denotes $\sum_{i \in V(C)}w_i$.
  For forests $F,F' \in \cf_n$ such that $F$ can be obtained from $F'$ by deleting an edge $uv$,
  observe that $\mass{F'} = \frac{\lambda}{\nu} \cdot \mass{F} \cdot w_uw_v$.
  For such $F, F'$ we let
\begin{equation}
    \varphi(F',F) = \frac{\nu}{\lambda} \cdot \frac{ \mass {F'} }{\cross_w(F)}. \label{eq:phi_def}
\end{equation}
  For all other pairs $F,F'$, we let $\varphi(F',F)=0$.
  
  For $i=1,\ldots,n$, let $\cf_{n}^{i}$ be the set of forests in $\cf_n$ with $i$ components.
  For each $F \in \cf_{n}^{i+1}$ we have
\[
  \sum_{F'\in\cf_{n}^{i}}\varphi(F',F) = \frac{\mass{F}}{\cross_w(F)} \sum_{uv \in \Cross(F)} w_uw_v = \mass{F};
\]
  %
  and thus for each $i=1,\ldots,n-1$
\begin{equation} \label{eqn.sumflow}
 		\sum_{F'\in\cf_{n}^{i}}\sum_{F\in\cf_{n}^{i+1}}\varphi(F',F) =
 		\sum_{F\in \cf_{n}^{i+1}} \mass{F} = K\cdot\pr(\RF \in \cf_{n}^{i+1})
\end{equation}
 as in Lemma 3.1 of~\cite{amr2011}.
      
  %
 Given a tree $T$ on $[n]$ and an integer $k$ with $1 \leq k \leq \lfloor W/2 \rfloor$,
 let $c(T,k)$ be the number of edges $e$ in $T$ such that $T \setminus e$ has a component with weight~$k$.
  %
    Let $\RT$ be $\RF$ conditioned on being a tree, so that $\RT$ is a random tree on $[n]$
    with $\pr(\RT =T) \propto \mass{T}$.
 %
   The distribution of $\RT$ is exactly as in~\cite{amr2011} -- 
   the weighting is not relevant here, since $e(T)=n-1$ and $\kappa(T)=1$ are fixed,
   and thus $\pr(\RT = T) \propto \prod_{i=1}^{n} w_i^{d_T(i)}$.
   Hence from Section 4 of~\cite{amr2011} we see that for any $\eta>0$, for $W$ sufficiently large we have
\begin{equation} \label{claim.half}
 \sum_{k \geq 1} \frac{\E{[c(\RT,k)]}}{k(W-k)} \leq (1+\eta) \cdot \frac12.
\end{equation}  
  By~(\ref{eqn.sumflow}) with $i=1$,
  corresponding to lemma 4.1 of~\cite{amr2011} we have  
\begin{eqnarray*}
    \pr(\RF \in \cf_{n}^{2}) &=&
      \frac{1}{K} \frac{\nu}{\lambda} \sum_{T \in \cf_n^1} \mass{T} \sum_{e \in T} \frac{1}{\cross_w(T-e)}\\
    &=&
      \pr(\RF \in\cf_{n}^{1}) \ \frac{\nu}{\lambda} \sum_{T \in \cf_n^1} \pr(\RT =T) \sum_{k \geq 1}
      \frac{c(\RT,k)}{k(W-k)}
\end{eqnarray*}
  and so
\begin{equation} \label{eqn.fn2_bound}
      \pr(\RF \in \cf_{n}^{2}) = \pr(\RF \in\cf_{n}^{1})\ \frac{\nu}{\lambda}\ \sum_{k \geq 1}
      \frac{\E{[c(\RT,k)}]}{k(W-k)}.
\end{equation}
   From~(\ref{claim.half}) and~(\ref{eqn.fn2_bound}) we see that
   for all $\eta > 0$, for $W$ sufficiently large, for all
   $w_1,\ldots,w_n$ with $\sum_{j=1}^n w_j=W$,
 \begin{equation} \label{eqn.ratio_1_2}
        \pr(\RF \in \cf_{n}^{2}) \leq (1+\eta) \frac{\nu}{2 \lambda} \ \pr(\RF \in \cf_{n}^{1}).
  \end{equation}
            
  Now we can complete the proof of~(\ref{eqn.impliesmain3})
  (and thus of~(\ref{eqn.asympt-comps}) in Theorem~\ref{prop.asympt-conn}) 
  as follows, as in the proof of Claim 2.2 in~\cite{amr2011}.
  The next lemma will allow us to assume that $n$ is large, as well as being useful later.
  %
  Using~(\ref{eqn.sumflow}) and the proof of Lemma 3.2 of~\cite{amr2011} we find:
\begin{lemma}\label{lem.smalln2.5} 
   For each $i=1,\ldots,n-1$
   \begin{equation}\label{eqn.smalln2.5}
       \pr(\RF \in \cf_{n}^{i+1}) \leq \frac{\pr(\RF \in \cf_{n}^{i})}{i} \frac{n}{W} \frac{\nu}{\lambda}.
   \end{equation}
\end{lemma} 
 If $W \geq 2n$ then by the above result and Lemma~\ref{lem.po2},
 $\kappa(\RF) \leq_s 1+ \Po(\frac{\nu}{2 \lambda})$ and so~(\ref{eqn.impliesmain3}) holds.
 Thus we may assume from now on 
 that $W<2n$.
 
 Next we introduce Lemma 3.3 of~\cite{amr2011}.
 For each finite non-empty set $V$ of positive integers,
  let $\cg_V$ denote the set of all graphs on the vertex set $V$,
  and let $\cg^k_V$ denote the set of all graphs in $\cg_V$ with exactly
  $k$ components.
  %
  For each positive integer $n$, let $\mu_n$ be a measure on the set of all graphs
  with vertex set a subset of $[n]$ which is \emph{multiplicative on components};
  that is, if $G$ has components $H_1,\ldots,H_k$ then
  $\mu_n(G)=\prod_{i=1}^k \mu_n(H_i)$. 
  Observe that we obtain such a measure
  if we set $\mu_n(G) = \mass{G}$ when $G$ is a forest and $\mu_n(G)=0$ otherwise. 
\begin{lemma}\label{lem.connect_1} (\cite{amr2011}) 
  Suppose there exist $\alpha > 0$ and integers $n \geq m_0 \geq 1$ such that
  \begin{equation}\label{eq.connect_1_a}
    \mu_n(\cg^2_V) \leq \alpha \mu_n(\cg^1_V) \;\;\; \mbox{for all }
    V \subseteq \{1,\ldots,n\} \mbox{ with } |V| \geq m_0.
  \end{equation}
  Then for all integers $k \geq 1$ and $n \geq km_0$
  \begin{equation}
    \label{eq:connect_1}
    \mu_{n}(\cg_{[n]}^{k+1})\leq \frac{\alpha}{k}\mu_n(\cg_{[n]}^k).
  \end{equation}
\end{lemma}


  We may now complete the proof of~(\ref{eqn.impliesmain3}). 
  %
  Fix $j \geq t$ large enough that
\[ \sum_{i \geq j} \left(\frac{\nu}{\lambda}\right)^i \frac{1}{i!} \leq \frac{\eps}{2}.\]
  Fix $\eta > 0$ small enough that, with $\alpha=(1+\eta) \frac{\nu}{2 \lambda}$,
  we have
\[ \pr(\Po(\alpha) \geq t) \leq \pr(\Po(\frac{\nu}{2 \lambda}) \geq t) + \eps/2. \]
  %
  %
  By~(\ref{eqn.ratio_1_2}) and Lemma \ref{lem.connect_1}
  it follows that, for $W$ large enough (recall that $n >W/2$), for all $i$ with $1\leq i \leq j$ we have
\[  \pr(\kappa(\RF)=i+1) \le \frac{\alpha}{i} \pr(\kappa(R^f)=i). \]
  In terms of $X= \kappa(\RF)-1$, this says that
\begin{equation} \label{eqn.krf}
  \pr(X=i+1) \le \frac{\alpha}{i+1} \pr(X=i) \;\; \mbox{ for } i=0,1,\ldots,j-1.
\end{equation} 
%
  Also, by Lemma \ref{lem.smalln2.5}, for all $i \geq 1$ 
\[ \pr(\RF \in \cf_{n}^{i+1})\leq \left(\frac{n \nu}{W \lambda}\right)^i \frac{1}{i!}
   \leq \left(\frac{\nu}{\lambda}\right)^i \frac{1}{i!}\]
  and so it follows by our choice of $j$ that
\[ \pr(X \geq j) = \pr(\kappa(\RF) \geq j+1) \leq \eps/2. \]
  Hence by~(\ref{eqn.krf}) and Lemma~\ref{lem.po1},
\begin{eqnarray*}
 \pr(X \geq t)
 & = & \pr(j-1 \geq X \geq t) + \pr(X \geq j)\\
 & \leq & \pr(\Po(\alpha) \geq t) + \eps/2\\
  & \leq & \pr(\Po(\frac{\nu}{2\lambda}) \geq t) + \eps.
\end{eqnarray*}
%


  This completes the proof of~(\ref{eqn.impliesmain3}) and thus of~(\ref{eqn.asympt-comps}).
  The inequality~(\ref{eqn.asympt-conn}) follows directly from~(\ref{eqn.asympt-comps}),
  so it remains only to prove~(\ref{eqn.asympt-Ekappa}).  Let $\eps>0$.  By Theorem~\ref{prop.tauconn},
  if $Y \sim \Po(\nu/\lambda)$ then
\[ \E[\kappa(R_n) {\bf 1}_{\kappa(R_n) \geq t+1}] \leq \E[(1+Y) {\bf 1}_{Y \geq t}] < \eps/2\]
  if $t$ is sufficiently large.  Fix such a $t$.
  By~(\ref{eqn.impliesmain2}) (applied for each value up to~$t$, and with $\eps$ replaced by $\frac{\eps}{2t}$), for $n$ sufficiently large
\begin{eqnarray*}
  \E[\kappa(R_n) {\bf 1}_{\kappa(R_n) \leq t}]
  &=& 
  \sum_{i=1}^{t} \pr(t \geq \kappa(R_n)\geq i)\\
  & \leq & \sum_{i=0}^{t-1} \pr(\Po(\frac{\nu}{2 \lambda})\geq i) + \eps/2\\
  & \leq & 1+ \E[\Po(\frac{\nu}{2 \lambda})] + \eps/2 = 1 + \frac{\nu}{2 \lambda} + \eps/2.
\end{eqnarray*}  
  Hence, for $n$ sufficiently large, $\E[\kappa(R_n)] \leq 1 + \frac{\nu}{2 \lambda}+ \eps$, and we are done.
   We have now completed the proof of Theorem~\ref{prop.asympt-conn}.


\subsection{Proof of Theorem~\ref{prop.fragbound2}}
\label{sec.fragproof}

  Let $\eps>0$, and let $\ca'_n = \{G \in \ca_n:\frag(G) \leq \eps n\}$.  Then $\ca'_n$ is bridge-addable.
  Also, for each $G \in \ca'_n$ we have $\cross(G) \geq (1-\eps)n \cdot \frag(G)$.
  Hence by Lemma~\ref{lem.frag1} with $\beta=1-\eps$ we have
\[  \E[\frag(R_n) {\bf 1}_{\frag(R_n) \leq \eps n}] < (1-\eps)^{-1} (\nu/\lambda). \]
  Thus it suffices for us to show that
\begin{equation} \label{eqn.frageps}
  \E[\frag(R_n) {\bf 1}_{\frag(R_n) > \eps n}] \; = \; o(1) \;\; \mbox{  as } n \to \infty.
\end{equation}
  
  We proceed as in the proof of Theorem~\ref{prop.asympt-conn}. 
  For a graph $G$ on $[n]$ let $\wfrag(G)$ denote $W$ minus the maximum weight $w(C)$ of a component $C$ of~$G$.
  Then corresponding to~(\ref{eqn.stocheq}) we have
\[ \frag (\RG) =_s \wfrag(\RF). \]
  To see this we may argue as in the proof of~(\ref{eqn.stocheq}), 
  recalling that for each $H \in \cb =[G_0]$, $g(H)$ is the forest obtained by contracting the components $C_i$ of $G_0$, and  
  noting that we have $\frag(H)=\wfrag(g(H))$.
  Thus it suffices to consider $\RF$
  and show that $\E[X {\bf 1}_{X > \eps W}]$ is $o(1)$ as $W \to \infty$,
  where $X = \wfrag(\RF)$. 
  \smallskip

  Define $\ttau(F)$ to be $\tau(g^{-1}(F))$ for each $F \in \cf_n$; and for $\ca \subseteq \cf_n$ let
  $\ttau(\ca) = \sum_{F \in \ca} \ttau(F) = \tau(g^{-1}(\ca))$.
  Then $\ttau(\cf_n)=\tau(\cb)$, and $\ttau(\ct_n)=\tau(\cc)$ where $\cc$ is the set of connected graphs in $\cb$.
  Thus by Theorem~\ref{prop.tauconn} we have
  $\ttau(\ct_n) \geq e^{-\alpha} \ttau(\cf_n)$, where we let $\alpha$ denote $\nu / \lambda$.
  Recall that we take $f(G) \equiv 1$ without loss of generality, and so
  $\tau(H)= \alpha \lambda^n$ for each connected graph $H$ in $\cb$.
  Also
\[ |g^{-1}(\ct_n)| = \prod_{i=1}^{n} w_i \cdot W^{n-2}\]
   where $W=\sum_{i=1}^n w_i$. 
   This counting result goes back to Moon~\cite{moon67} in 1967 and R\'enyi~\cite{renyi1970} in 1970
   (see also Pitman~\cite{pitman99}) and appears in 
   the proof of Lemma 4.2 of~\cite{amr2011}.  Hence
\[ \ttau(\ct_n) = \alpha \lambda^n \prod_{i=1}^{n} w_i \cdot W^{n-2}.\]
  For each non-empty set $A$ of positive integers, let $\ct_A$ and $\cf_A$ denote respectively the sets of
  trees and forests on vertex set $A$, and let $W_A$ denote $\sum_{i \in A} w_i$.
  Let $1 \leq a \leq n-1$ and let $A \subseteq [n]$ with $|A|=a$.  Denote $[n] \setminus A$ by $\bar{A}$. Then
\begin{eqnarray*}
    \ttau(\cf_A) \ttau(\cf_{\bar{A}})
    & \leq & e^{2 \alpha}  \ttau(\ct_A) \ttau(\ct_{\bar{A}})\\
    &=& \alpha^2 \lambda^n e^{2 \alpha} (\prod_{i=1}^{n} w_i ) \cdot W_A^{a-2} (W-W_A)^{n-a-2}\\
    &=& \alpha^2 \lambda^n e^{2 \alpha} (\prod_{i=1}^{n} w_i) \cdot (W_A(W-W_A))^{-2} W_A^{a} (W-W_A)^{n-a}.
\end{eqnarray*}
  But $x^a(W-x)^{n-a}$ is maximised at $x=\frac{a}{n}W$, so
\[ W_A^{a} (W-W_A)^{n-a} \leq (\frac{a}{n} W)^{a} (\frac{n-a}{n} W)^{n-a} = a^a (n-a)^{n-a} \left(\frac{W}{n} \right)^n.\]
  Thus
\[  \ttau(\cf_A) \ttau(\cf_{\bar{A}}) \leq \alpha^2 \lambda^n e^{2 \alpha} \prod_{i=1}^{n} w_i \cdot (W_A(W-W_A))^{-2} a^a (n-a)^{n-a} \left(\frac{W}{n} \right)^n.\]
  Now by Stirling's formula, there are positive constants $c_1$ and $c_2$ such that for each positive integer $k$
\[  c_1 k^{k+\frac12} e^{-k} \leq k! \leq c_2 k^{k+\frac12} e^{-k}. \]
  Thus
\begin{eqnarray*}
  \sum_{a=1}^{n-1} {n \choose a} a^{a} (n-a)^{n-a}
  &=& n! \sum_{a=1}^{n-1} \frac{a^a}{a!} \frac{(n-a)^{n-a}}{(n-a)!}\\
  & \leq & c_1^{-2} n! e^n \sum_{a=1}^{n-1} (a(n-a))^{-\frac12}.
\end{eqnarray*}
  But the sum in this last expression is $O(1)$, so 
  %
\[  \sum_{a=1}^{n-1} {n \choose a} a^{a} (n-a)^{n-a} \; \leq \; c_3 n! e^n \; \leq \; c_4 n^{n+\frac12} \]
  for some constants $c_3$ and $c_4$. It follows that
\[ \sum_{ A \subseteq [n]} (W_A W_{\bar{A}})^2 \ \ttau(\cf_A) \ttau(\cf_{\bar{A}})  \leq
    c_5 \alpha^2 \lambda^n e^{2 \alpha} (\prod_{i=1}^{n} w_i) W^n n^{\frac12} = c_6 \, \ttau(\ct_n) W^2 n^{\frac12} \]
  for some constant $c_5$, and $c_6 = c_6(\alpha) = c_5 \alpha e^{2 \alpha}$.
  Let us introduce a piece of notation: for $z>0$ let 
\[ \sum_{z \leq W_A \leq W/2} s_A := \sum \{ W_A \ \ttau(\cf_A) \ttau(\cf_{\bar{A}}): A \subseteq [n], z \leq W_A \leq W/2\}. \]
  Then
 \begin{equation} \label{eqn.end}
  \sum_{z \leq W_A \leq W/2} s_A \leq  z^{-1} (\frac{W}{2})^{-2} \sum_{ A \subseteq [n]} (W_A W_{\bar{A}})^2 \ 
  \ttau(\cf_A) \ttau(\cf_{\bar{A}}) 
  \leq 4 c_6 \ttau(\ct_n) z^{-1} n^{\frac12}.
\end{equation}
  Now we argue as in the proof of Proposition 5.2 of~\cite{cmcd08}.
  Consider a disconnected graph $G$ on $[n]$, and denote $\wfrag(G)$ by $z$.
  We claim that there is a union of components with weight in the interval $[z/2, W/2]$.
  To see this let $b=W-z$, so that $b$ is the biggest weight of a component.
  Note that $\lfloor \frac{W+b}{2} \rfloor - \lceil \frac{W-b}{2} \rceil +1 \geq b$,
  and so there are at least $b$ integers in the list
  $\lceil \frac{W-b}{2} \rceil,\ldots,\lfloor \frac{W+b}{2} \rfloor$.
 %
  Thus by considering adding components one at a time we see that there is a union of components, with vertex $A$ say,
  such that $W_A$ is in this set.  Then $W_A \geq \lceil \frac{W-b}{2} \rceil \geq z/2$ and
  $W-W_A \geq W - \lfloor \frac{W+b}{2} \rfloor \geq z/2$.  Thus $A$ or $\bar{A}$ is as required.
  
  From the above, there is an injection from the set of forests $F \in \cf_n$ with $\wfrag(F) \geq z$
  to the set
  of triples $A, F_A, F_{\bar{A}}$ where $A \subseteq [n]$, $z/2 \leq \wfrag(F)/2 \leq W_A \leq W/2$,
  $F_A \in \cf_A$, $F_{\bar{A}} \in \cf_{\bar{A}}$ and where $\ttau(F)= \ttau(F_A) \ttau(F_{\bar{A}})$.
  It follows that
\begin{eqnarray*}
  \ttau(\cf_n) \, \E[X 1_{X \geq z}] &=&
     \sum_{F \in \cf_n: {\small \wfrag}(F) \geq z} \ttau(F) \wfrag(F)\\
  & \leq &
     \sum_{A \subseteq [n], z/2 \leq W_A \leq W/2} \; \sum_{F_A \in \cf_A} \sum_{F_{\bar{A}} \in \cf_{\bar{A}}}
     2 W_A \ \ttau(F_A) \ttau(F_{\bar{A}})\\
  & = &    2 \sum_{z/2 \leq W_A \leq W/2} s_A.
\end{eqnarray*} 
  Hence
\[ \E[X 1_{X > \eps W}] \leq \frac{2}{\ttau(\cf_n)} \sum_{\eps W/2 \leq W_A \leq W/2} s_A =O(W^{-\frac12}) \]
  by~(\ref{eqn.end}), and the proof of Theorem~\ref{prop.fragbound2} is complete.


\section{Proofs for rooted graphs}
\label{sec.rooted}

  In this section we prove Theorems~\ref{prop.rooted-forests} and~\ref{prop.rooted} and on rooted graphs.
  First we prove Theorem~\ref{prop.rooted},
  following the lines of the proof of Theorem~\ref{prop.tauconn}.
  \smallskip
  
  \begin{proofof} {Theorem~\ref{prop.rooted}}
  As before, it suffices to assume that $\ca$ is $\ca_n$ for some $n$.
  Let $1\leq k \leq n-1$.
  Let $\cp$ be the set of pairs $(G^o,e)$ where $G^o$ is a rooted graph on $[n]$ with $k$ components and $e$ is a bridge
  in $G^{o}$ (which we may think of as being oriented towards the root of the component).
  Let $\cq$ be the set of pairs $(H^{o},uv)$, where $H^o$ is a rooted graph on $[n]$ with $k+1$ components, and $uv$ is an
  ordered pair of vertices such that $u$ is the root of its components and $v$ is in a different component.

  There is a natural bijection between $\cp$ and $\cq$.  Given $(G^o,e) \in \cp$, if $u$ is
  the end of $e$ further from the root, we delete the edge $e$ and make $u$ a new root: given $(H^{o},uv) \in \cq$,
  we add an edge between $u$ and $v$ and no longer have $u$ as a root.
  Further, if the $k+1$ components of $H^{o}$ have $n_1,\ldots, n_{k+1}$ vertices respectively,
  then the number of pairs $uv$ such that $(H^{o},uv) \in \cq$ is
  $\sum_{i=1}^{k+1} (n-n_i) = kn$. 
  
  Now much as in the proof of Theorem~\ref{prop.tauconn} we have  
\begin{eqnarray*}
   \tau(\ca_n^{k \, o}) \cdot(n\!-\!k)
    & \geq &
  \sum_{(G^{o},e) \in \cp}  \tau(G^{o}) \; \geq \; \frac{\lambda}{\nu}
  \sum_{(H^{o},uv) \in \cq} \tau(H^{o})\\
  & = & \frac{\lambda}{\nu} \
    \tau(\ca_n^{k+1 \, o}) \cdot k n.
\end{eqnarray*}
  Therefore
\begin{equation} \label{eqn.rooted}
   \tau(\ca_n^{k\!+\!1 \, o}) \leq
   \frac{n\!-\!k}{n} \frac{\nu}{\lambda k} \ \tau(\ca_n^{k \, o}).
\end{equation}
  %
  Thus for $R^{o} \in_{\tau} \cao$
\[
  \pr(\kappa(R^{o}) = k+1) \leq \frac{\nu}{\lambda k} \pr(\kappa(R^{o}) = k) \;\; \mbox{ for each } k=1,2,\ldots,  \]
  and we may complete the proof as for Theorem~\ref{prop.tauconn}.
\end{proofof}
  \bigskip
  
\begin{proofof}{Theorem~\ref{prop.rooted-forests}}
  Consider rooted forests 
  and let $R^{o}_n \in_{\tau} \cf^{o}$.
  Then the first two inequalities above hold at equality, so
\[  
  \pr(\kappa(R^{o}_n) = k+1) = \frac{n-k}{n} \frac{\nu}{\lambda k} \pr(\kappa(R^{o}_n) = k) \;\; \mbox{ for each } k=1,2,\ldots.
\]
  Hence $\kappa(R^{o}_n)$ converges in distribution to $1+\Po(\frac{\nu}{\lambda})$ as $n \to \infty$. 
  Thus to complete the proof of Theorem~\ref{prop.rooted-forests}, it remains only to show that
  $\E[\frag(R_n^{o})] \to \infty$ as $n \to \infty$.
  By what we have just proved and the fact that $|\ct_n^{o}|=n^{n-1}$ 
\[ \tau(\cf_n^{o}) \sim e^{\nu/\lambda} \ \tau(\ct_n^{o}) = \nu e^{\nu/\lambda} (\lambda n)^{n-1}. \] 
  We now obtain a lower bound on $\E[\frag(R_n^{o})]$ by considering forests with two components.
  With sums over say $\log n < j < n/2$, and using Stirling's formula, we have
\begin{eqnarray*}
  \E[\frag(R_n^{o})] & \geq & \tau(\cf_n^{o})^{-1} \cdot \sum_j {n \choose j} j \ \tau(\ct_j^{o}) \tau(\ct_{n-j}^{o})\\
  & \sim & \left(\nu e^{\nu/\lambda}\right)^{-1} (\lambda n)^{-(n-1)} n!
           \sum_j j \frac{(\lambda j)^{j-1}}{j!} \frac{(\lambda(n-j))^{n-j-1}}{(n-j)!}\\ 
  &=& \Theta(1) \cdot \frac{n!}{n^{n-1}} \sum_j \frac{j^j}{j!} \frac{(n-j)^{n-j-1}}{(n-j)!}\\
  &=& \Theta(1) \cdot n^{\frac32} e^{-n} \sum_j j^{-\frac12} e^j (n-j)^{-\frac32} e^{n-j}\\
  &=& \Theta(1) \cdot \sum_j j^{-\frac12} \; = \; \Theta(n^{\frac12}),
\end{eqnarray*}  
  and we are done.
\end{proofof}

  \bigskip

  \noindent
  {\em Aside on the unrooted case}
  
  From the inequality~(\ref{eqn.rooted}) above we may quickly deduce the result in Theorem~\ref{prop.tauconn} that,
  for $R \in_{\tau} \ca$ (where $\ca$ is finite and bridge-addable and not rooted) we have
  $\; \pr(R \mbox{ is connected}) \geq e^{-\nu/\lambda}$.  
  To see this note that we may assume as usual that $\ca$ is $\ca_n$,
  and note also that each graph $G \in \ca_n$ with $\kappa(G)=k+1$
  yields at least $n-k$ rooted graphs in $\cao_n$. 
  Now let $\cc_n$ be the set of connected graphs in $\ca_n$, and use~(\ref{eqn.rooted}) once and then $k-1$ further times:
  we find
\begin{eqnarray*} 
  \tau(\ca_n^{k+1}) 
  &\leq &
     \frac1{n-k}  \tau(\ca_n^{k\!+\!1 \, o}) 
   \; \leq \;
    \frac1{n} \frac{\nu}{\lambda k} \tau(\ca_n^{k \, o}) \\ 
   &\leq &
    \frac1{n} \left(\frac{\nu}{\lambda}\right)^k \frac1{k!} \tau(\cc_n^{o}) \; = \;
   \left(\frac{\nu}{\lambda}\right)^k \frac1{k!} \tau(\cc_n).
\end{eqnarray*}
  Thus
\[ \tau(\ca_n) \leq \sum_{k \geq 0} \left(\frac{\nu}{\lambda}\right)^k \frac1{k!} \tau(\cc_n) = e^{\nu/\lambda} \tau(\cc_n),\]
  and the proof is complete.
  

\section{Concluding remarks}
\label{sec.concl}

  Consider $R_n \inu \ca$ where $\ca$ is bridge-addable.
  Starting from the lower bound $e^{-1}$ 
  on the probability that $R_n$ is connected and the stronger stochastic bound
  $\kappa(R_n) \leq_s 1+ \Po(1)$ on the number of components, 
  it was natural to enquire to what extent the bounds could be improved to match known results for forests.
  Our results suggest that we should think of these bounds as being out asymptotically by a factor~2
  in the `parameter', in that the ratio $\lambda/\nu$ should be doubled
  (though the corresponding bounds are tight in the rooted case).
  
  The central conjecture on connectivity for a bridge-addable set of graphs is from~\cite{msw06},
  and was re-stated here as Conjecture~\ref{conj.b-add}.
  As we noted earlier, some asymptotic improvement has been made on the bound $e^{-1}$
  \cite{bbg07,bbg10,sn2012}; and the full improvement to $e^{-\frac12}$ has been achieved,
  but only when we make the stronger assumption that $\ca$ is bridge-alterable~\cite{amr2011,kp2011}.

  In the present paper we considered corresponding improvements concerning the distribution of $\kappa(R_n)$,
  and introduced new bounds on $\frag(R_n)$ (the number of vertices left when we remove 
  a largest component)
  which also match results for forests asymptotically.  Further, we set these results in a general framework
  emphasising the role of bridges, rather than just considering uniform distributions.
  
  For random rooted graphs our non-asymptotic results already match those for forests.
  In each other case, to achieve results asymptotically matching those for forests we have had to assume that the set
  $\ca$ of graphs is bridge-alterable.
  It is natural to ask whether these results actually hold under the weaker assumption that $\ca$
  is bridge-addable.  Conjecture~\ref{conj.b-add} is still open. 
  We propose two further conjectures (for uniform distributions).
  The first concerns a possible extension of inequalities~(\ref{eqn.b-alt-frag2}) and~(\ref{eqn.b-alt-comps1})
  from bridge-alterable to bridge-addable.  
  
\begin{conjecture} \label{conj.b-add-comps}
  If $\ca$ is bridge-addable and $R_n \in_u \ca$ then
 \[ \mbox{(a)} \hspace{.1in} \kappa(R_n) \leq_s 1+ \Po(\frac12) \;\;\mbox{ asymptotically},\]
  and 
 \[ \mbox{(b)} \hspace{.1in} \limsup_{n \to \infty} \E[\frag(R_n)] \leq 1. \hspace{.9in}\]
\end{conjecture}
  The work of Balister, Bollob{\'a}s and Gerke~\cite{bbg07,bbg10} mentioned earlier
  gives some progress on part (a) of this conjecture: the proofs there
  together with Lemma~\ref{lem.po1} here show that, with $\alpha=0.7983$,
\[ \kappa(R_n) \leq_s 1+ \Po(\alpha) \;\;\mbox{ asymptotically};\]
  and this bound gives
\[ \limsup_{n \to \infty} \E[\kappa(R_n)] \leq 1+ \alpha \approx 1.7983 \]
  as may be seen by arguing as at the end of the proof of Theorem~\ref{prop.asympt-conn}.
  Note also that to establish part (b) of the conjecture,
  it suffices to show that~(\ref{eqn.frageps}) holds whenever $\ca$ is bridge-addable.
  \smallskip
  
  Finally, we propose a strengthened non-asymptotic version of the last conjecture,
  along the lines of Conjecture 5.1 in~\cite{amr2011} or Conjecture 1.2 of~\cite{bbg10}.
  %
\begin{conjecture} \label{conj.b-add-strong}
  If $\ca$ is bridge-alterable, $n$ is a positive integer, $R_n \in_u \ca$ and $R_n^{\cf} \inu \cf$ then
\[ \kappa(R_n) \leq_s \kappa(R_n^{\cf})\;\; \mbox{ and } \;\; \E[\frag(R_n)] \leq \E[\frag(\RF_n)]. \]
\end{conjecture}  
  Establishing the stronger version of this conjecture, in which we assume only that $\ca$ is bridge-addable, would 
  of course be even better!


\end{document}